\documentclass[a4paper,10pt]{article}
\pdfoutput=1

\usepackage[utf8]{inputenc}

\usepackage{amsmath,amsfonts,amsthm,amssymb}
\usepackage{graphicx}
\usepackage[colorlinks,citecolor=blue,linkcolor=blue]{hyperref}
\usepackage[nameinlink,noabbrev,capitalize]{cleveref}
\usepackage{tabularx}
\usepackage{booktabs}
\usepackage{enumitem}

\numberwithin{equation}{section}

\usepackage[
	style=numeric,
    url=true,
	doi=true,                
	isbn=false,
	eprint=true,
 	giveninits=true,         
	maxbibnames=99,          
	maxcitenames=2,
	backend=biber,           
	safeinputenc,            
	sortcites=true,
]{biblatex}

\renewbibmacro*{doi+eprint+url}{%
    \printfield{doi}%
    \newunit\newblock%
    \iftoggle{bbx:eprint}{%
        \usebibmacro{eprint}%
    }{}%
    \newunit\newblock%
    \iffieldundef{doi}{%
        \usebibmacro{url+urldate}}%
        {}%
    }

\newtheorem{theorem}{Theorem}[section]
\newtheorem{lemma}[theorem]{Lemma}
\newtheorem{definition}[theorem]{Definition}

\newtheorem{corollary}[theorem]{Corollary}

\newtheorem{assumption}{Assumption}

\newcommand\be{\begin{equation}}
\newcommand\ee{\end{equation}}

\renewcommand{\subset}{\subseteq}

\renewcommand{\d}{\,\text{\rmfamily{}\upshape{}d}}
\newcommand{\dx}{\,\text{\rmfamily{}\upshape{}d}x}

\def\N{\mathbb  N}
\def\R{\mathbb  R}

\def\gph{\operatorname{gph}}

\def\supp{\operatorname{supp}}

\def\dist{\operatorname{dist}}
\def\diag{\operatorname{diag}}

\def\Otwo{\operatorname{O}(2)}
\def\SOtwo{\operatorname{SO}(2)}

\newcolumntype{L}{>{$}l<{$\quad}}
\newcolumntype{R}{>{$}r<{$\quad}}
\newcolumntype{C}{>{$}c<{$}}

\begin{document}

\title{Control in the coefficients of an elliptic differential operator: topological derivatives and Pontryagin maximum principle}

\author{Daniel Wachsmuth%
\thanks{Institut f\"ur Mathematik,
Universit\"at W\"urzburg,
97074 W\"urzburg, Germany, {\tt daniel.wachsmuth@mathematik.uni-wuerzburg.de}.
This research was partially supported by the German Research Foundation DFG under project grant Wa 3626/5-1.}}

\maketitle

{\bfseries Abstract.}
We consider optimal control problems, where the control appears in the main part of the operator.
We derive the Pontryagin maximum principle as a necessary optimality condition.
The proof uses the concept of topological derivatives. In contrast to earlier works,
we do not need continuity assumptions for the coefficient or gradients of solutions of partial differential equations.
Following classical proofs, we consider perturbations of optimal controls by multiples of characteristic functions of sets,
whose scaling factor is send to zero.
For $2d$ problems, we can perform an optimization over the elliptic shapes of such sets leading
to stronger optimality conditions involving a variational inequality of a new type.

\bigskip

{\bfseries Keywords. } Optimal control, control in the coefficients, Pontryagin maximum principle, topological derivatives

\bigskip

{\bfseries MSC (2020) classification. }
49K20, 
35J15  
\section{Introduction}

In this article, we are interested in proving the Pontryagin maximum principle
maximum principle for the following problem:
Minimize
\begin{equation}\label{eq_001}
J(y,a):=\frac12 \int_\Omega (y(x)-y_d(x))^2 \dx + \int_\Omega g( a(x)) \dx
\end{equation}
over all
\begin{equation}\label{eq_002}
 a \in  \mathcal A \subset L^\infty(\Omega; \R^{d,d}),
\end{equation}
where $y\in H^1_0(\Omega)$ is the weak solution of
\begin{equation}\label{eq_003}
-\operatorname{div} (a \nabla y) = f \text{ a.e. in } \Omega.
\end{equation}
Hence, the optimization variable is the coefficient in the main part of the differential operator.
In this problem, $\Omega \subset \R^d$ is a bounded domain, $f\in H^{-1}(\Omega)$ is a given source term, $y_d \in L^2(\Omega)$
is the desired state, while $g: \R^{d,d} \to \R \cup \{+\infty\}$ models the cost of choosing a certain coefficient matrix.
In addition, $\mathcal A$ is a feasible set, that contains matrices with uniformly positive definite symmetric part.
For the precise statement of the assumptions, we refer to \cref{sec_setup}.

Problem \eqref{eq_001}--\eqref{eq_003} is a classical problem, and lead to the study of H-convergence, \cite{MuratTartar1997}.
One cannot prove existence of solutions, and examples without solutions can be found, e.g., in \cite{Allaire2002,Murat1977}.

Here, we are interested in proving the Pontryagin maximum principle,
which is a classical necessary optimality condition in optimal control theory.
Let $a$ be a solution of the problem above. Then we consider a perturbation of the type
\begin{equation}\label{eq_intro_perturbation}
 a_r := a + \chi_{r \omega} (b-a),
\end{equation}
where $b \in \R^{d,d}$ is a constant matrix, $\omega$ is the unit ball in $\R^d$.
And we are interested in passing to the limit in the difference quotient
\begin{equation}\label{eq_intro_dq}
 \frac1{ r^d } ( J(y_r,a_r) - J(y,a)),
 \end{equation}
where $y_r$ is the solution to the elliptic equation with coefficient $a_r$.
For the state-dependent part of the cost functional $J$, we have the following expansion
\begin{multline*}
\frac12 \|y_r-y_d\|_{L^2(\Omega)}^2 - \frac12 \|y-y_d\|_{L^2(\Omega)}^2 \\
= - \int_\Omega (a_r-a) \nabla y \cdot\nabla p \dx  - \int_\Omega (a_r-a) \nabla y\cdot \nabla (\tilde p_r-p)\dx,
\end{multline*}
where $p$ and $\tilde p_r$ are certain adjoint states.
The first term in the expansion represents the Fr\'echet derivative of the map $a\mapsto y$ from $L^\infty(\Omega)$ to $H^1_0(\Omega)$,
while the second term is of higher-order in $\|a_r-a\|_{L^\infty(\Omega)}$.
Due to the choice of the perturbation, we do not get that $\|a_r-a\|_{L^\infty(\Omega)} \to 0$ for $r\to0$.
And the second term in the expansion does not vanish when passing to the limit in the difference quotient \eqref{eq_intro_dq}.

As one would expect, we will encounter topological derivatives of solutions of the elliptic partial differential equations.
In contrast to earlier works, we prove the corresponding results
under much weaker assumptions than in the literature:
\begin{enumerate}
 \item The coefficient function $a$ is assumed to belong to $L^\infty(\Omega; \R^{d,d})$ such that $-\operatorname{div}(a\nabla \cdot)$
 is a uniformly elliptic operator. We do not assume that $a$ is piecewise constant \cite{Allaire2002,AmmariKang2007,Amstutz2021}
 or continuous \cite{GanglSturm2020},
 \item We work with weak solutions in $H^1$. We do not assume that the weak solutions $y$ or their gradients $\nabla y$ are continuous.

 \item That is, our proof works under the same set of assumptions than the Lax-Milgram theorem, and the proof only uses regularity of solutions
 as provided by Lax-Milgram.
\end{enumerate}
The assumption of piecewise constant coefficients $a$ makes sense in material or topology optimization.
However, this assumption is too restrictive for the optimization problem \eqref{eq_001}--\eqref{eq_003}.
Our proof follows the developments of \cite{Amstutz2021,GanglSturm2020,Sturm2020_semilinear}. The main improvement compared to these earlier works
is the consequent use of the celebrated Lebesgue differentiation theorem, which allows to dispense with continuity assumptions.
This derivation is done in \cref{sec_sensitivity} with the main result being the asymptotic expansion of the cost functional in \cref{thm_exp_general}.

We interested in proving necessary optimality conditions for \eqref{eq_001}--\eqref{eq_003}.
Here, we need to specify the notion of local solutions.

\begin{definition}
A coefficient function $a$ is called locally optimal in $L^p(\Omega)$ for \eqref{eq_001}--\eqref{eq_003}, $p \in [1,\+\infty]$,  with associated state $y$, if there
is $\rho>0$ such that $J(y,a) \le J( \tilde y,\tilde a)$ for all feasible coefficients $\tilde a$ with $\|a-\tilde a\|_{L^p(\Omega)} <\rho$
and associated state $\tilde y$.
\end{definition}

If the feasible set $\mathcal A$ is bounded in $L^\infty(\Omega)$, then local optimality in $L^p(\Omega)$ implies local optimality in $L^q(\Omega)$ for all $1\le p < q \le \infty$.

We are mainly interested in the case $p=1$.
Here, we obtain the necessary optimality conditions in form of the Pontryagin maximum principle.
Let $a$ be (locally) optimal in $L^1(\Omega)$.
Then the Pontryagin maximum principle is satisfied in the following way:
for almost all $x_0 \in \Omega$ and all feasible perturbations $b\in \R^{d,d}$ we have
\[\begin{split}
 0 \le
 -(b-a(x_0)) \nabla y(x_0) \cdot \left(\nabla p(x_0) + \frac1{|\omega|} \int_\omega \nabla_{x'} Q \dx' \right)
 + g(b) - g(a(x_0)).
\end{split}\]
Here, $Q$ is a solution of a certain adjoint equation on $\R^d$, which depends on $a(x_0)$, $b$, and the set $\omega$, which is used
in the perturbation \eqref{eq_intro_perturbation}.
The precise statement can be found in \cref{thm_pmp}.

If $a$ would only be locally optimal in $L^\infty(\Omega)$
then one could use Fr\'echet derivatives, and one would obtain a similar inequality but with $Q \equiv 0$.
For certain shapes $\omega$ (balls in $\R^d$, ellipses in $\R^2$), explicit formulas for $Q$ are available.
Similar results can be found in the work of Raitums \cite{Raitum1984a,Raitum1984b,Raitum1986,Raitum1989},
which deserves to be better known, but which seems to be only available in Russian.

In the special case that the coefficient $a$ is scalar and $d=2$, one can optimize the above formula for elliptic shapes $\omega$
to obtain the following strengthened version
 \begin{multline*}
-(b-a(x_0)) \nabla y(x_0) \cdot \nabla p(x_0)+  g(b) - g(a(x_0)) \\
+ \frac12 \frac{ (b-a(x_0))^2}b \left(   \nabla y(x_0) \cdot \nabla p(x_0)- \|\nabla y(x_0)\|_2\|\nabla p(x_0)\|_2\right)
 \ge 0.
 \end{multline*}
A related result can be found in \cite{Amstutz2011b} for the study of a material optimization problem, where $a$ is allowed to take
only two different values.
These inequalities are stronger than the related inequalities one gets using Fr\'echet derivatives in $L^\infty(\Omega)$.

The plan of the paper is as follows. The sensitivity analysis of the cost functional $J$ with respect
to perturbations of the coefficient is performed in \cref{sec_sensitivity}, where the main result is \cref{thm_exp_general}.
The special cases of perturbations with characteristic functions of balls and ellipses are considered in \cref{sec_specific_perturbations}.
These results are applied in \cref{sec_pmp} to an optimal control problem with control in the coefficients.

\paragraph{Notation}

Given $v\in \R^d$, we denote its Euclidean norm by $\|v\|_2$.
The set $B(x,r)\subset\R^d$ is the open ball centered at $x$ with radius $r$.
The characteristic function of a set $A\subset \R^d$ is denoted by $\chi_A$.
The Lebesgue measure of a measurable set $A\subset \R^d$ is denoted by $|A|$,
the measure of a ball with radius $r$ in $\R^d$ is denoted by $|B(r)|$.

As is customary in the literature on partial differential equations, we will denote
the inner product in $\R^d$ of gradients by dots, i.e., $\nabla y(x) \cdot \nabla p(x) := \nabla p(x)^T \nabla y(x)$.

\section{Sensitivity analysis with respect to perturbations on general sets} \label{sec_sensitivity}

\subsection{Setup of the problem}
\label{sec_setup}

Throughout the paper we assume the following about the data of the problem.

\begin{assumption}\label{ass_basic}
Let $\Omega\subset \R^d$ be a bounded domain.
Let $\alpha>0$ be given.
In addition, let $y_d,f\in L^2(\Omega)$ be given.
\end{assumption}

Let us define the set of admissible coefficient functions by
\begin{equation}\label{eq_def_calA}
 \mathcal A := \{ a \in L^\infty(\Omega; \R^{d,d}): \ a(x)\in \mathcal M \ \text{ f.a.a. } x\in \Omega\},
\end{equation}
where
\begin{equation}\label{eq_def_M}
 \mathcal M := \{ a \in \R^{d,d}: \ \xi^T a \xi \ge \alpha|\xi|^2 \ \forall \xi\in \R^d\}.
\end{equation}
In the sequel, we will work with coefficient functions from the set $\mathcal A$. Note that
$a(x)\in \mathcal A$ for almost all $x\in \Omega$ is the minimum requirement in order that the Lax-Milgram theorem guarantees existence of weak solutions.
We will not assume more regularity of $a$ and $\Omega$, and we will not rely on any elliptic regularity results beyond basic $H^1$-regularity.

Let us fix a reference coefficient $a\in \mathcal A$. We denote the corresponding solution of the state equation by $y$,
i.e., $y\in H^1_0(\Omega)$ solves
\begin{equation}\label{eq_def_y}
 \int_\Omega a \nabla y \cdot \nabla v \dx = \int_\Omega f v\dx \quad \forall v\in H^1_0(\Omega).
\end{equation}
Here, we used the notation
\[
 a \nabla y \cdot \nabla v := \sum_{i,j=1}^d a_{ij} \frac\partial{\partial x_j} y \frac\partial{\partial x_i}v.
\]
By Lax-Milgram theorem, the equation \eqref{eq_def_y} is uniquely solvable.
Given $a$, we define
\[
 J(a):= \frac12 \|y-y_d\|_{L^2(\Omega)}^2,
\]
where $y$ is the solution of \eqref{eq_def_y} to $a\in A$. We are interested in the sensitivity analysis of $J$ with respect to
perturbations of $a$. Here, we will use perturbations by characteristic functions, which is a well-known concept in optimal control, inverse problems, or material optimization.
To this end, let $\omega\subset \R^d$ be an open bounded set with $0\in \omega$. Given a point $x_0\in \Omega$, a value $b \in \mathcal M$, and a radius (or scaling parameter) $r>0$,
we define
\begin{equation}\label{eq_def_ar}
 a_r := a + \chi_{x_0 + r\omega} (b-a).
\end{equation}
The goal of this section is to compute the variation of $J$ at $a$ with respect to $b$, $\omega$, which is defined as
\[
 \delta J(a;b,x_0,\omega) := \lim_{r\to0} \frac{ J(a_r) - J(a) }{ r^d|\omega| }.
\]
Note that $r^d|\omega|$ is the Lebesgue measure of $r \omega$, and it is larger or equal to the $L^0$- or Ekeland distance between $a_r$ and $a$.

The exposition in the following subsections follows earlier work \cite{GanglSturm2020,Sturm2020_semilinear,Amstutz2021}.
As one would expect, the statements of the main results are identical. However, we
do not use or assume any regularity beyond $L^\infty$ for the coefficients and $H^1$ for the weak solutions of state and adjoint equations.

\subsection{Basic expansion of the functional}

Recall the definition of $a_r$ in \eqref{eq_def_ar}.
Let $y_r$ be the corresponding solution of the state equation, i.e.,
\begin{equation}\label{eq_def_yr}
 \int_\Omega a_r \nabla y_r \cdot \nabla v \dx = \int_\Omega f v\dx \quad \forall v\in H^1_0(\Omega).
\end{equation}
Note that the difference $y_r-y$ satisfies the following equation
\begin{equation}\label{eq_diff_yr}
 \int_\Omega a_r \nabla (y_r-y) \cdot \nabla v \dx + \int_\Omega (a_r-a) \nabla y \cdot \nabla v \dx=  0 \quad \forall v\in H^1_0(\Omega).
\end{equation}
We define the averaged adjoint $\tilde p_r\in H^1_0(\Omega)$, see, e.g., \cite{Sturm2020_semilinear},
as the solution of
\begin{equation}\label{eq_def_aadj}
   \int_\Omega a_r \nabla v \cdot \nabla \tilde p_r \dx = \frac12 \int_\Omega [(y_r-y_d)+(y-y_d)] v \dx \quad \forall v\in H^1_0(\Omega).
\end{equation}
In addition, let the adjoint $p\in H^1_0(\Omega)$ be given as the solution of
\begin{equation}\label{eq_def_p}
  \int_\Omega a \nabla v \cdot \nabla p \dx =  \int_\Omega (y-y_d) v \dx \quad \forall v\in H^1_0(\Omega).
\end{equation}
Then we have the following result.

\begin{lemma}\label{lem_expand}
Let $a_r$ as in \eqref{eq_def_ar} with the notation from \cref{sec_setup}.
Then it holds
\[
J(a_r) - J(a) = - \int_\Omega (a_r-a) \nabla y \cdot\nabla p \dx  - \int_\Omega (a_r-a) \nabla y\cdot \nabla (\tilde p_r-p)\dx ,
\]
where $y$, $\tilde p_r$, $p$ solve \eqref{eq_def_y}, \eqref{eq_def_aadj}, \eqref{eq_def_p}.
\end{lemma}
\begin{proof}
Using the equations \eqref{eq_def_aadj} and \eqref{eq_diff_yr}, we find
\[\begin{split}
 J(a_r) - J(a) &= \frac12 \|y_r-y_d\|_{L^2(\Omega)}^2 - \frac12 \|y-y_d\|_{L^2(\Omega)}^2 \\
 & = \frac12( y_r - y, \ (y_r-y_d) + (y-y_d)) \\
 & = \int_\Omega a_r \nabla (y_r - y) \cdot \nabla \tilde p_r \dx \\
 & =-\int_\Omega (a_r-a) \nabla y \cdot\nabla \tilde p_r\dx \\
 & =- \int_\Omega (a_r-a) \nabla y \cdot\nabla p \dx  - \int_\Omega (a_r-a) \nabla y \cdot\nabla (\tilde p_r-p)\dx ,
 \end{split}
\]
which is the claim.
\end{proof}
Here the second term in the expansion of \cref{lem_expand} seems to be of second order.
In fact, it corresponds to the remainder term in a Taylor expansion of $J$ using the Fr\'echet differentiability of $a \mapsto y$ from
$L^\infty(\Omega)$ to $H^1_0(\Omega)$. Consequently it is of the order $\|a_r-a\|_{L^\infty(\Omega)}^2$.
However, the second term is {\em not} of higher order with respect to $r\searrow0$.
In fact, the second term converges with the same order as the first for $r\searrow0$.

The key for the asymptotic analysis of the expansion of \cref{lem_expand}
is the following coordinate transform.
This well-known idea is to transform the small set $x_0 + r\omega$ to $\omega$.
To this end, let us define the following functions:
\begin{equation}\label{eq_def_Tr}
 T_r(x') :=  x_0 + rx', \quad T_r^{-1}(x) := (x-x_0)/r.
\end{equation}

%

\subsection{Lebesgue differentiation theorem}

Before continuing with the analysis of the expansion, let us recall the Lebesgue differentiation theorem
in the following form.

\begin{theorem}\label{thm_lebesgue_lp}
 Let $u\in L^p_{loc}(\R^d)$ for some $p\in [1,\infty)$.
Then
\[
 \lim_{r\to0} \frac1{|B(x,r)|} \int_{B(x,r)} |u(y)-u(x)|^p \d y =0
\]
for almost all $x\in \R^d$.
These points $x$ are called  {\em $p$-Lebesgue points} of $u$.
\end{theorem}
\begin{proof}
See \cite[Section 1.7, Corollary 1]{EvansGariepy1992}.
\end{proof}
Note that the claim is slightly different from the standard formulation of the theorem given by
\[
 \lim_{r\to0} \frac1{|B(x,r)|} \int_{B(x,r)} u(y) \d y = u(x)
\]
for almost all $x\in \R^d$. In the sequel, the Lebesgue differentiation theorem in the form of \cref{thm_lebesgue_lp}
will be a valuable tool. It will serve as a replacement of continuity assumptions frequently encountered in the literature.

A well-known result says that if there is $C>0$ and $\alpha \in (0,1]$ such that
\[
  \frac1{|B(x,r)|} \int_{B(x,r)} |u(y)-u(x)| \d y  \le C r^\alpha
\]
for all $x$, then $u$ is H\"older continuous with order $\alpha$.
For a precise formulation, see \cite[Theorem 4.3]{RafeiroSamkoSamko2013}.
This might explain the heavy use of H\"older continuity assumptions in the literature on topological derivatives
for perturbations in the coefficients of the differential operator.
As we will show, convergence to zero as in  \cref{thm_lebesgue_lp} is enough, no faster convergence with respect to $r\searrow0$ is needed.

It is well-known that
the above theorem can be generalized to take means on sets of bounded eccentricity. Here, we will use the following modification.

\begin{corollary}\label{cor_lebesgue_lp_general}
Let $\omega \subset \R^d$ be such that there are $\rho_1,\rho_2>0$ with
\[
 B(0,\rho_1) \subset \omega \subset  B(0,\rho_2).
\]
Let $x$ be a $p$-Lebesgue point of $u$.
Then
\[
 \lim_{r\to0} \frac1{r^d|\omega|} \int_{x + r\omega} |u(y)-u(x)|^p \d y =0.
\]
\end{corollary}
\begin{proof}
 The claim follows from
 \[
  \frac1{r^d|\omega|} \int_{x + r\omega} |u(y)-u(x)|^p \d y
  \le \frac1{|B(x,r\rho_1)|} \int_{B(x,\rho_2r)} |u(y)-u(x)|^p \d y \to0,
 \]
see also \cite[Section 3.1.2]{SteinShakarchi2005}.
\end{proof}

The interplay between the Lebesgue differentiation theorem and the coordinate transform \eqref{eq_def_Tr} is made precise in the next result.

\begin{corollary}\label{cor_lebesgue_lp}
 Let $u \in L^p(\Omega)$, $p\in [1,\infty)$.
 Let $x_0 \in \Omega$ be a $p$-Lebesgue point of $u$.
 Let $T_r$ be given by \eqref{eq_def_Tr}.
 Then
 \[
  \lim_{r\to0}\int_{s \omega} |u\circ T_r - u(x_0)|^p \dx' = 0
 \]
 for all $s\ge0$.
\end{corollary}
\begin{proof}
By elementary calculations, we get
\[
  \int_{s\omega} |u\circ T_r - u(x_0)|^p \dx' = r^{-d} \int_{x_0 + sr \omega} |u(x)-u(x_0)|^p \dx
  \to 0,
\]
where we have used \cref{cor_lebesgue_lp_general}.
\end{proof}

\subsection{Transformed linearized state equation}

In this section, we will investigate the asymptotics of $\frac 1{r^d|\omega|} (y_r-y)$.
Let us recall from  \eqref{eq_diff_yr}
that the difference $y_r-y$ satisfies
\[
   \int_\Omega a_r \nabla (y_r-y) \cdot\nabla v \dx + \int_\Omega (a_r-a) \nabla y \cdot\nabla v \dx=  0
\quad \forall v\in H^1_0(\Omega).
\]
Following earlier works, e.g., \cite{Sturm2020_semilinear}, we define
\begin{equation}\label{eq_def_Kr}
 K_r(x') := \frac1r (y_r-y)\circ T_r(x').
\end{equation}
 Due to construction, it follows $\nabla_{x'} K_r(x') = ((\nabla (y_r-y)) \circ T_r)(x')$.
Then $K_r$ satisfies the transformed equation
\begin{equation}\label{eq_Kr}
 \int_{T_r^{-1}(\Omega)} a_r \circ T_r \nabla_{x'} K_r \cdot \nabla_{x'} v \dx' + \int_\omega ((b-a)\nabla y)\circ T_r \cdot \nabla_{x'} v \dx' =0
\end{equation}
for all $v\in H^1_0(T_r^{-1}( \Omega))$. Note, $K_r\in H^1_0(T_r^{-1}( \Omega))$ implies that its extension by zero belongs to $H^1(\R^d)$ without
assumptions on the regularity of $\Omega$, see \cite[Lemma 3.22]{Adams1975}.
Hence, in the sequel we will assume $K_r$ is extended by zero to $\R^d$.
The next goal is to pass to the limit $r\searrow0$ in \eqref{eq_Kr}.
We have the following bound of $K_r$.

\begin{lemma}\label{lem_Kr_bounded}
 Let $K_r$ be given by \eqref{eq_def_Kr}.
 Let $x_0$ be a $2$-Lebesgue point of $|(b-a) \nabla y|$.
 Then there is $M>0$ such that
 \[
 r^{2} \|K_r\|_{L^2(\R^d)}^2 + \|\nabla_{x'} K_r\|_{L^2( \R^d )}^2 \le M.
 \]
\end{lemma}
\begin{proof}
Testing \eqref{eq_diff_yr} with $y_r -y $ yields
\[\begin{split}
 \alpha \|\nabla (y_r-y)\|_{L^2(\Omega)}^2 \dx & \le -\int_\Omega (a_r-a)\nabla y \cdot \nabla (y_r-y) \dx\\
 &\le \frac\alpha2  \|\nabla (y_r-y)\|_{L^2(\Omega)}^2 + \frac1{2\alpha} \int_{x_0 + r \omega} |(b-a) \nabla y|^2 \dx.
\end{split}\]
By Poincare inequality, we get
\[
 \|y_r - y\|_{L^2(\Omega)}^2 + \|\nabla (y_r-y)\|_{L^2(\Omega)}^2 \le c \int_{x_0 + r \omega} |(b-a) \nabla y|^2 \dx
\]
for some $c>0$ only depending on $\Omega$ and $\alpha$.
Due to \cref{thm_lebesgue_lp}, we have that
\[
 \frac1{r^d|\omega|}  \int_{x_0 + r \omega} |(b-a) \nabla y|^2 \dx
\]
converges for $r\searrow0$. This shows that there is $M>0$ such that
\[
  \|y_r - y\|_{L^2(\Omega)}^2  + \|\nabla (y_r-y)\|_{L^2(\Omega)}^2 \le M r^d|\omega|
\]
for all $r>0$ sufficiently small.
Applying the coordinate transform $T_r$ to this inequality and using the definition of $K_r$ in \eqref{eq_def_Kr} proves the claim.
%
\end{proof}

\begin{corollary}\label{cor_rKr_weakly_zero}
 Let $K_r$ be given by \eqref{eq_def_Kr}.
 Let $x_0$ be a $2$-Lebesgue point of $|(b-a) \nabla y|$.
 Then
 \[
  r K_r \rightharpoonup 0 \text{ in } L^2(\R^d).
 \]
\end{corollary}
\begin{proof}
 By \cref{lem_Kr_bounded}, we have that $(rK_r)_{r>0}$ is uniformly bounded in $H^1(\R^d)$ with
 $r \nabla_{x'} K_r \to 0$ in $L^2(\R^d)$ for $r\searrow0$.
 Let $r_k \searrow0$ such that $r_k K_{r_k} \rightharpoonup w$ in $L^2(\R^d)$.
 It follows $\nabla_{x'} w=0$ in $\R^d$, so that $w$ has to be equal to a constant function.
 Since $w\in L^2(\R^d)$, it follows $w=0$.
 \end{proof}

Let us define the coefficient function
\begin{equation}\label{def_atilde}
 \tilde a := \chi_\omega b + \chi_{\omega^c} a(x_0).
\end{equation}
It turns out that $\tilde a$ is the limit of the transformed coefficients $a_r \circ T_r$ in the following sense.

\begin{lemma}\label{lem_conv_coeff}
Let $x_0$ be a $2$-Lebesgue point of $a$.
We have the convergence
\[
 a_r \circ T_r \to \tilde a
\]
in $L^2(n\omega)$ for all $n\in \N$.
\end{lemma}
\begin{proof}
Observe that $ a_r \circ T_r = \tilde a$ on $\omega$. Then we get
\[
 \int_{n\omega} |a_r \circ T_r - \tilde a|^2 \dx' = \int_{n\omega \setminus \omega} |(a \circ T_r)(x') - a(x_0)|^2 \dx',
\]
and the claim is a direct consequence of \cref{cor_lebesgue_lp}.
\end{proof}

Let us introduce the so-called homogeneous Sobolev space (or Beppo-Levi space).
Here, we define it as the space of equivalence classes under the relation $u \sim v \Leftrightarrow \nabla(u-v)=0$.
That is, the constant functions are quotiented out:
\[
 \dot H^1(\R^d) := \{ z\in H^1_{\mathrm{loc}}(\R^d): \ \nabla z \in L^2(\R^d)\} \ / \  \R.
\]
It is a Hilbert space when supplied with the norm $\|z\|_{ \dot H^1(\R^d)}:=\|\nabla z\|_{L^2(\R^d)}$.
In addition, equivalence classes of test functions from $C_c^\infty(\R^d)$ are dense in $\dot H^1(\R^d)$. 
For the proofs, we refer \cite[Section 3]{Praetorius2004}.
Note that \cite{Praetorius2004} denotes the Beppo-Levi space $\dot H^1(\R^d)$ by $H^1(\R^d)$, \cite[eq.\, (10)]{Praetorius2004}.

This space is the proper space to look for the weak limit of $K_r$ for $r\searrow0$.

\begin{lemma}\label{lem_conv_Kr_K}
 Let $x_0$ be a $2$-Lebesgue point of $|(b-a) \nabla y|$ and $a$.
Then we have $K_r \rightharpoonup K$ in $\dot H^1(\R^d)$, where $K$ is the unique solution of
\begin{equation}\label{eq_K}
  \int_{\R^d} \tilde a \nabla_{x'} K \cdot \nabla_{x'} v \dx' + (b-a(x_0))\nabla y(x_0) \cdot \int_\omega \nabla_{x'} v \dx' =0 \quad \forall v\in \dot H^1(\R^d).
\end{equation}
\end{lemma}
\begin{proof}
 Existence and uniqueness of the solution of \eqref{eq_K} follows directly from the Lax-Milgram theorem.
 For each $r>0$, the function $K_r$ belongs to the spaces $H^1_0(T_r^{-1}(\Omega))$ and $H^1(\R^d)$, hence $K_r\in \dot H^1(\R^d)$.
By \cref{lem_Kr_bounded}, $(K_r)$ is bounded in $\dot H^1(\R^d)$.

Let $r_k\searrow0$ be a sequence such that $K_{r_k} \rightharpoonup \tilde K$ in $\dot H^1(\R^d)$.
It remains to pass to the limit in the equation \eqref{eq_def_Kr}. Let $v\in C_c^\infty(\R^d)$ be given. Let $n\in \mathbb N$ be such that $n\omega  \supseteq \supp v$.
Let $\rho>0$ be such that $\rho n < \dist(x_0, \partial \Omega)$. Then $T_\rho( \supp v )\subset \Omega$, and $v $ can be used as test function in \eqref{eq_Kr} if $r<\rho$.
Due to \cref{lem_conv_coeff}, we have $a_r \circ T_r \to \tilde a$ in $L^2(n\omega)$.
This allows us to pass to the limit in the first integral of  \eqref{eq_Kr}:
\[
\int_{T_{r_k}^{-1}(\Omega)} a_{r_k} \circ T_{r_k} \nabla_{x'} K_{r_k} \cdot \nabla_{x'} v \dx'  \to
\int_{\R^d} \tilde a \nabla_{x'} \tilde K \cdot \nabla_{x'} v \dx'.
\]
The convergence of the second integral follows from \cref{cor_lebesgue_lp}, and we have
\[
 \int_\omega ((b-a)\nabla y)\circ T_{r_k} \cdot \nabla_{x'} v \dx' \to (b-a(x_0))\nabla y(x_0) \cdot \int_\omega \nabla_{x'} v \dx' .
\]
Since test functions from $C_c^\infty(\R^d)$ are dense in $\dot H^1(\R^d)$, it follows that $\tilde K$ satisfies \eqref{eq_K}, hence $\tilde K=K$.
\end{proof}

The convergence of $K_r$ to $K$ is even strong.
A similar result can be found in \cite[Proposition 4.1]{Amstutz2021} and \cite[Theorem 4.3]{GanglSturm2020},
where they assumed continuity of $\nabla y$ at $x_0$. In addition, the reference coefficient $a$ was assumed to be constant,
so that (in our notation) $a_r \circ T_r = \tilde a$, compare also \cref{lem_conv_coeff}.
In order to deal with the convergence $a_r \circ T_r \to \tilde a$, we follow an idea of \cite[Theorem 3.1]{CasasFernandez1993}.

\begin{lemma}\label{lem_conv_Kr_K_strong}
Under the assumptions of \cref{lem_conv_Kr_K}, we have
$ K_r \to  K$ in $\dot H^1(\R^d)$.
\end{lemma}
\begin{proof}
We will use the Cholesky decomposition.
Let $M^+$ denote the set of matrices with positive definite symmetric part, i.e.,
\[
 M^+ := \{ A\in \R^{d,d}: \ \frac12 (A+A^T) \text{ positive definite }\}.
\]
Then $M^+$ is open in $\R^{d,d}$.
We will denote the map from the symmetric part of a matrix to its lower triangular Cholesky factor by $L$,
i.e., $L(A)$ is such that $\frac12 (A+A^T) = L(A)L(A)^T$ for matrices $A \in M^+$. 
The map $A\mapsto L(A)$ is continuous from $M^+$ to $\R^{d,d}$.
In addition, we have $\|L(A)\|_F^2 = \operatorname{tr}(L(A)L(A)^T)=\operatorname{tr}(A)$,
so that the superposition operator induced by $L$ is nicely behaved.
In particular, it is continuous from $L^2(\R^d)$ to $L^4(\R^d)$.

Since all relevant matrices are elements of $\mathcal M$, see \eqref{eq_def_M}, we have the following bound on $L(A)^{-1}$.
For $A\in \mathcal M$, we have that $\|L(A)^{-1}\|_F^2 = \operatorname{tr}( 2 (A+A^T)^{-1}) $.
The eigenvalues of $\frac12(A+A^T)$ are bounded from below by $\alpha$, so that $\|L(A)^{-1}\|_F^2 \le d\alpha^{-1}$.
In addition, $A\mapsto L(A)^{-1}$ is continuous on $\mathcal M$.

By extending $a_r \circ T_r$  with zero to $\R^d$, we can consider it as function on the domain $\R^d$.
We are going to use the equality $a \nabla v \cdot \nabla v = \frac12 (a+a^T)\nabla v \cdot \nabla v = | L(a)^T\nabla v|^2$.
Testing \eqref{eq_Kr} with $K_r$, we get
\[
 \int_{\R^d} |L(a_r \circ T_r)^T \nabla_{x'} K_r|^2 \dx' = - \int_\omega ((b-a)\nabla y)\circ T_r \cdot \nabla_{x'} K_r \dx'.
\]
Due to \cref{cor_lebesgue_lp} and the weak convergence $\nabla_{x'} K_r \rightharpoonup \nabla_{x'}K$ in $L^2(\R^d)$,
the right-hand side converges as
\[\begin{split}
 - \int_\omega ((b-a)\nabla y)\circ T_r \cdot \nabla_{x'} K_r \dx'
 & \to -(b-a(x_0))\nabla y(x_0) \cdot \int_\omega \nabla_{x'} K \dx'\\
 &=  \int_{\R^d} |L(\tilde a)^T \nabla_{x'} K|^2 \dx'
\end{split}\]
where we have used \eqref{eq_K}. Hence, $ L(a_r \circ T_r)^T \nabla_{x'} K_r$ is bounded in $L^2(\R^d)$, and in addition $\|L(a_r \circ T_r)^T \nabla_{x'} K_r\|_{L^2(\R^d)}$
converges to $\|L(\tilde a)^T \nabla_{x'} K\|_{L^2(\R^d)}$.

Due to \cref{lem_conv_coeff}, $L(a_r \circ T_r)^T$ converges in $L^4(n\omega)$ to $L(\tilde a)^T$.
Together with the weak convergence $\nabla_{x'} K_r \rightharpoonup \nabla_{x'} K$, we get that the weak limit of
$L(a_r \circ T_r)^T \nabla_{x'} K_r$ in $L^2(\R^d)$ is equal to $L(\tilde a)^T \nabla_{x'} K$.
Due to the convergence of the norms, it follows
\[
L(a_r \circ T_r)^T \nabla_{x'} K_r \to  L(\tilde a)^T \nabla_{x'} K  \text{ in }L^2(\R^d).
\]
We will prove the desired convergence with the celebrated dominated convergence theorem.
Let $r_k\searrow0$ be a sequence.
Then after extracting a subsequence if necessary, we have the pointwise convergence
$L(a_{r_k} \circ T_{r_k})^T \nabla_{x'} K_{r_k} \to  L(\tilde a)^T \nabla_{x'} K$ a.e.\@ on $\R^d$.
Applying a diagonal sequence argument to the result of \cref{lem_conv_coeff}, we can extract another subsequence (still denoted the same)
such that $a_{r_k} \circ T_{r_k} \to \tilde a$ a.e.\@ on $\R^d$.
Using the continuity of $a \mapsto L(a)^{-T}$, we get the pointwise a.e.\@ convergence
\[
 \nabla_{x'} K_{r_k} = L(a_{r_k} \circ T_{r_k})^{-T} L(a_{r_k} \circ T_{r_k})^T \nabla_{x'} K_{r_k}
 \to   L(\tilde a)^{-T} L(\tilde a)^T \nabla_{x'} K =\nabla_{x'} K.
\]
Since the $L^\infty$-norms of $L(a_{r_k} \circ T_{r_k})^{-T}$
are uniformly bounded, the claim follows with the dominated convergence theorem applied to
\begin{multline*}
\nabla_{x'} K_{r_k} - \nabla_{x'} K =  L(a_{r_k} \circ T_{r_k})^{-T} ( L(a_r \circ T_r)^T \nabla_{x'} K_r -  L(\tilde a)^T \nabla_{x'} K) \\
+( L(a_{r_k} \circ T_{r_k})^{-T} L(\tilde a)^T - I) \nabla_{x'} K.
\end{multline*}
The convergence for $r \to0$ follows by a standard subsequence-subsequence argument.
\end{proof}

\subsection{Transformed  adjoint equations}

Let us recall the definition of the averaged adjoint $\tilde p_r$,
\[
  \int_\Omega a_r \nabla v \cdot \nabla \tilde p_r \dx = \frac12 \int_\Omega [(y_r-y_d)+(y-y_d)] v \dx \quad \forall v\in H^1_0(\Omega).
\]
Let us define
\begin{equation}\label{eq_def_Qr}
 Q_r(x') := \frac1r (\tilde p_r-p)\circ T_r(x'),
\end{equation}
where $p$ satisfies the adjoint equation
\[
  \int_\Omega a \nabla v \cdot \nabla p \dx = \int_\Omega (y-y_d) v \dx \quad \forall v\in H^1_0(\Omega).
\]
\begin{lemma} \label{lem_Qr_bounded}
Let $x_0$ be a $2$-Lebesgue point of $(b-a) \nabla p$.
 Let $Q_r$ be given by \eqref{eq_def_Qr}. Then there is $M>0$ such that
 \[
 r^2 \|Q_r\|_{L^2(\R^d)}^2 + \|\nabla_{x'} Q_r\|_{L^2( \R^d )}^2 \le M
 \]
 for all $r>0$ small enough.
\end{lemma}
\begin{proof}
 We proceed as in the proof of \cref{lem_Kr_bounded}. Note that the difference $\tilde p_r-p$ satisfies
 \[
  \int_\Omega a_r \nabla v \cdot \nabla (\tilde p_r-p) \dx + \int_\Omega (a_r-a)\nabla v\cdot \nabla p \dx = \frac12 \int_\Omega (y_r-y) v \dx \quad \forall v\in H^1_0(\Omega).
 \]
Testing this equation with $\tilde p_r-p$, we obtain
 using Poincare inequality and standard estimation procedures
 \[
\|\tilde p_r-p\|_{L^2(\Omega)}^2 + \|\nabla(\tilde p_r-p)\|_{L^2(\Omega)}^2
  \le c \left(  \int_{x_0 + r\omega} |(b-a) \nabla p|^2 \dx + \int_\Omega (y-y_r)^2 \dx\right),
 \]
 where $c$ is independent of $r$.
Due to \cref{thm_lebesgue_lp} and \cref{lem_Kr_bounded}, the right-hand side is bounded by $M r^d |\omega|$.
The proof follows using the definition of $Q_r$.
\end{proof}

\begin{lemma}\label{lem_conv_Qr_Q}
Let $x_0$ be a $2$-Lebesgue point of $(b-a) \nabla p$ and $a$.
We have $Q_r \rightharpoonup Q$ in $\dot H^1(\R^d)$, where $Q$ is the unique solution of
\begin{equation}\label{eq_Q}
  \int_{\R^d} \tilde a\nabla_{x'} v  \cdot \nabla_{x'} Q \dx' + (b-a(x_0)) \int_\omega \nabla_{x'} v \dx'  \cdot \nabla p(x_0) =0 \quad \forall v\in \dot H^1(\R^d).
\end{equation}
\end{lemma}
\begin{proof}
 The proof follows the lines of the proof of \cref{lem_conv_Kr_K}.
 Unique solvability of \eqref{eq_Q} follows from Lax-Milgram theorem. After applying the coordinate transform, we find that
 $Q_r$ satisfies
\begin{equation}\label{eq_Qr}
   \int_{T_r^{-1}(\Omega)} a_r \circ T_r\nabla_{x'} v \cdot \nabla_{x'} Q_r  \dx' + \int_\omega ((b-a)\circ T_r) \nabla_{x'} v \cdot (\nabla p)\circ T_r\dx' = \int_{T_r^{-1}(\Omega)} rK_r v\dx
 \end{equation}
for all $v\in H^1_0(T_r^{-1}( \Omega))$.

Let $v\in C_c^\infty(\R^d)$ be given. Then for $r>0$ small enough the above equation is fulfilled. Passing to the limit in the integrals on the left-hand side of \eqref{eq_Qr}
can be done exactly as in the proof of \cref{lem_conv_Kr_K}.
The integral on the right-hand side vanishes for $r\searrow0$ due to \cref{cor_rKr_weakly_zero}.
And the claim is proven.
\end{proof}

Strong convergence  $Q_r \to Q$ in $\dot H^1(\R^d)$ can be proven similarly as in \cref{lem_conv_Kr_K_strong}.

\begin{lemma}\label{lem_cont_b_to_Q}
Let $a\in \mathcal M$, $p\in \R^d$ be given.
For $b\in \mathcal M$, let $Q(b) \in \dot H^1(\R^d)$ denote the weak solution of
\begin{equation}\label{eq_Q_b}
  \int_{\R^d} (a + \chi_\omega(b-a) )\nabla_{x'} v  \cdot \nabla_{x'} Q(b) \dx' + (b-a ) \int_\omega \nabla_{x'} v \dx'  \cdot p =0 \quad \forall v\in \dot H^1(\R^d).
\end{equation}
Then the map $b\mapsto Q(b)$ is continuous from $\mathcal M$ to $\dot H^1(\R^d)$.
\end{lemma}
\begin{proof}
Let $b_1,b_2 \in \mathcal M$ be given. Using $v := Q(b_2)$ as test function in the equation \eqref{eq_Q_b} for $Q(b_2)$, we get the a-priori bound
\[
 \alpha \|\nabla Q(b_2)\|_{L^2(\R^d)}^2 \le \|b_2-a\|_2  \|p\|_2 \int_\omega \nabla_{x'} Q(b_2) \dx',
\]
which proves $\alpha \|\nabla Q(b_2)\|_{L^2(\R^d)} \le \|b_2-a\|_2  \|p\|_2 |\omega|^{1/2}$.
Define $Q:=Q(b_1) - Q(b_2)$.
Subtracting equations \eqref{eq_Q_b} for $Q(b_1)$ and $Q(b_2)$ we obtain
\begin{multline*}
  \int_{\R^d} (a + \chi_\omega(b_1-a) )\nabla_{x'} v  \cdot \nabla_{x'} Q \dx'
  +
  \int_\omega (b_1-b_2) \nabla_{x'} v  \cdot \nabla_{x'} Q(b_2) \dx' \\
  +
  (b_1-b_2 ) \int_\omega \nabla_{x'} v \dx'  \cdot p =0 \quad \forall v\in \dot H^1(\R^d).
\end{multline*}
Using $v:=Q$, we get
\[
 \alpha \|  \nabla Q\|_{L^2(\R^d)}^2 \le \|b_1-b_2\|_2  \| \nabla Q\|_{L^2(\R^d)}\left(\|  \nabla Q(b_2)\|_{L^2(\R^d)} + \|p\|_2 |\omega|^{1/2}
 \right),
\]
which proves the claim.
\end{proof}

\subsection{First sensitivity result}

\begin{theorem} \label{thm_exp_general}
Let $\omega\subset \R^d$ be an open bounded set with $0\in \omega$. Let $b\in \mathcal M$. Then for almost all $x_0\in \Omega$,
we have
\begin{multline*}
 \delta J(a;b,x_0,\omega) = \lim_{r\searrow0} \frac{ J(a_r) - J(a) }{ r^d|\omega| }\\
 = -(b-a(x_0)) \nabla y(x_0) \cdot \left(\nabla p(x_0) + \frac1{|\omega|} \int_\omega \nabla_{x'} Q \dx' \right) ,
\end{multline*}
where $Q\in \dot H^1(\R^d)$ is the solution of \eqref{eq_Q}.
The function $Q$ depends solely on $a(x_0)$, $b$, $\nabla p(x_0)$, and $\omega$.
\end{theorem}
\begin{proof}
 Let $x_0$ be a $1$-Lebesgue point of the integrable functions
 $(b-a) \nabla y\cdot \nabla p$,
 and a $2$-Lebesgue point of the $L^2$-functions
 $(b-a) \nabla y$, $(b-a)\nabla p$, and $a$.
Due to \cref{thm_lebesgue_lp}, the set of such points $x_0$ differs from $\Omega$ by a set of measure zero.

 We will use the expansion from \cref{lem_expand}
 \[
  J(a_r) - J(a) =- \int_\Omega (a_r-a) \nabla y \cdot\nabla p \dx  - \int_\Omega (a_r-a) \nabla y \cdot\nabla (\tilde p_r-p)\dx.
 \]
Due to \cref{thm_lebesgue_lp}, we have
\[
 \lim_{r\searrow0}\frac1{r^d|\omega|}\int_\Omega (a_r-a) \nabla y \cdot \nabla p \dx  = (b-a(x_0)) \nabla y(x_0) \cdot \nabla p(x_0).
\]
Using the coordinate transform $T_r$ and the definition of $Q_r$, cf., \eqref{eq_def_Tr} and \eqref{eq_def_Qr}, we have
\[
  \frac1{r^d|\omega|} \int_\Omega (a_r-a) \nabla y \cdot \nabla (\tilde p_r-p)\dx = |\omega|^{-1} \int_\omega ((a_r-a) \nabla y)\circ T_r \cdot \nabla_{x'} Q_r \dx'.
\]
Due to \cref{cor_lebesgue_lp}, the term $((a_r-a) \nabla y)\circ T_r$ converges in $L^2(\omega)$ to the constant function $(b-a(x_0))\nabla y(x_0)$.
By \cref{lem_conv_Qr_Q}, we have $\nabla_{x'} Q_r \rightharpoonup \nabla_{x'} Q$ in $L^2(\R^d)$. Hence, it follows
\[
 \lim_{r\searrow0} |\omega|^{-1} \int_\omega ((a_r-a) \nabla y)\circ T_r \cdot \nabla_{x'} Q_r \dx' = |\omega|^{-1} (b-a(x_0))\nabla y(x_0) \cdot \int_\omega \nabla_{x'} Q \dx',
\]
and the claim is proven.
\end{proof}

Using the definitions of $Q$ and $K$, we have the following identity
\begin{multline*}
 (b-a(x_0)) \nabla y(x_0) \cdot \int_\omega \nabla_{x'} Q \dx'
 = - \int_{\R^d} \tilde a \nabla_{x'} K  \cdot \nabla_{x'} Q \dx' \\
 = (b-a(x_0))\int_\omega \nabla_{x'} K \dx' \cdot \nabla p(x_0),
\end{multline*}
which can be used to equivalently rewrite the result of \cref{thm_exp_general}.

In case that $a(x_0)$ and $b$ are multiples of the indentity matrix, the formula in \cref{thm_exp_general}
can be written in terms of the so-called polarization matrix $M$, i.e,
\begin{equation}\label{eq_expansion_polariz}
  \delta J(a;b,x_0,\omega) = - \frac1{|\omega|} (b-a(x_0)) \frac{a(x_0)}b \nabla y(x_0) \cdot M \nabla p(x_0).
\end{equation}
The scaling of the polarization matrix is not unique across the literature, here we followed \cite[Theorem 2.1]{BruhlHankeVogelius2003}.
Many details on polarization matrices can be found in the monograph \cite{AmmariKang2007}.

\subsection{Sensitivity matrix}

For fixed $a(x_0)$ and $b$, the mapping $\nabla p(x_0) \mapsto \int_\omega \nabla_{x'} Q \dx'$ is a linear mapping from $\R^d$ to $\R^d$.
Hence, the mapping $ ( \nabla y(x_0), \nabla p(x_0) ) \mapsto (b-a(x_0)) \nabla y(x_0) \cdot \int_\omega \nabla_{x'} Q \dx'$
can be written in terms of a matrix multiplication.

Given $p\in \R^d$ let us define $Q_p \in \dot H^1(\R^d)$ as the solution of
\begin{equation}\label{eq_Qp}
  \int_{\R^d} \tilde a\nabla_{x'} v  \cdot \nabla_{x'} Q \dx' + (b-a(x_0)) \int_\omega \nabla_{x'} v \dx'  \cdot p =0 \quad \forall v\in \dot H^1(\R^d).
\end{equation}
Let us define the matrix $R\in \R^{d,d}$ by
\begin{equation}\label{eq_def_R}
 p^T R y := -(b-a(x_0)) y \cdot \int_\omega \nabla_{x'} Q_p \dx'.
\end{equation}
Of course, $R$ depends on the coefficients $a(x_0)$ and $b$.
The matrix $R$ has the following properties.

\begin{lemma}
 Suppose $a(x_0)$ and $b$ are symmetric. Then $R$ defined in \eqref{eq_def_R} is symmetric.
\end{lemma}
\begin{proof}
Let $Q_y\in \dot H^1(\R^d)$ be the solution of
\[
  \int_{\R^d} \tilde a\nabla_{x'} v  \cdot \nabla_{x'} Q \dx' + (b-a(x_0)) \int_\omega \nabla_{x'} v \dx'  \cdot y =0 \quad \forall v\in \dot H^1(\R^d).
\]
Then using the symmetry assumption and testing the equation for $Q_y$ with $Q_p$ results in
\[\begin{split}
p^T R y  &= -(b-a(x_0)) y \cdot \int_\omega \nabla_{x'} Q_p \dx' \\
&=\int_{\R^d} \tilde a\nabla_{x'} Q_p  \cdot \nabla_{x'} Q_y\dx' \\
&=-(b-a(x_0))  \int_\omega \nabla_{x'} Q_y \dx' \cdot p \\
&= y^TRp,
\end{split}\]
which proves the symmetry of $R$.
\end{proof}

\begin{lemma}
 Suppose $b-a(x_0)$ is symmetric. Then $p^TRp \ge0 $ forall $p$, and $p^TRp=0$ if and only if $(b-a(x_0))p=0$.
\end{lemma}
\begin{proof}
Take $p\in \R^d$.
Since $b-a(x_0)$ is symmetric, it follows
\[\begin{split}
   -(b-a(x_0)) p\cdot \int_\omega \nabla_{x'} v
    =-(b-a(x_0)) \int_\omega \nabla_{x'} v \dx'  \cdot p
    = \int_{\R^d} \tilde a\nabla_{x'} v  \cdot \nabla_{x'} Q_p \dx'
\end{split}\]
for all $v\in \dot H^1(\R^d)$,
which implies
\[\begin{split}
p^T R p  &=-(b-a(x_0)) p \cdot \int_\omega \nabla_{x'} Q_p \dx' \\
&=\int_{\R^d} \tilde a\nabla_{x'} Q_p  \cdot \nabla_{x'} Q_p\dx' \ge  \alpha \|\nabla_{x'} Q_p\|_{L^2(\R^d)}^2\ge  0
\end{split}\]
since $\tilde a(x') \in \mathcal M$ for almost all $x'$.

Clearly, $p^T R p =0$ if $(b-a(x_0)) p=0$.
Let us assume $p^T R p =0$. This implies $\nabla_{x'}Q_p=0$ and $(b-a(x_0)) \int_\omega \nabla_{x'} v \dx'  \cdot p =0$ forall $v\in \dot H^1(\R^d)$.
Setting $v(x') := p^T(b-a(x_0))x'$ proves $(b-a(x_0))p=0$, and  definiteness of $R$ is established as claimed.
\end{proof}

For anisotropic variants of the polarization matrix introduced in \eqref{eq_expansion_polariz}, we refer to \cite[Section 4.12]{AmmariKang2007} and \cite{KangKim2007}.
%


\section{Perturbations on balls and ellipses}
\label{sec_specific_perturbations}

Here, we will work in the isotropic case, that is,
$a(x_0)$ and $b$ are assumed to be positive multiples of the identity matrix.
In this case, we get explicit expressions for the variation $\delta J(a;b,x_0,\omega)$.
With a little abuse of notation, we will assume
\[
 a(x_0), b\in \R, \ a(x_0),b \ge \alpha.
\]
Polarization matrices for the anisotropic case and $d=2,3$ were computed in \cite{KangKim2007}.

\subsection{Balls}

Here, we will set
\[
 \omega := B(0,1).
\]
Interestingly, in this case the solutions of \eqref{eq_K} and \eqref{eq_Q} can be computed explicitly.

\begin{lemma}\label{lem_explicit_KQ}
Let $\omega := B(0,1)$.
Let $g\in \R^d$ be given, $a(x_0),b\in \R$. Then the solution $G\in \dot H^1(\R^d)$ of the equation
 \[
  \int_{\R^d} \tilde a \nabla_{x'} G \cdot \nabla_{x'} v \dx' + g \cdot \int_\omega \nabla_{x'} v \dx' =0 \quad \forall v\in \dot H^1(\R^d).
  \]
  is given  by
  \[
   G(x'):= g \cdot x'  \frac1{\max(1, |x'|^d)} \frac{-1}{b+a(x_0)(d-1)}.
  \]
\end{lemma}
\begin{proof}
Note that $\Delta G(x')=0$ for all $x'$ with $|x'|\ne1$.
Define
\[
 G_0(x') =  g \cdot x'  \frac1{\max(1, |x'|^d)} .
\]
Let $v\in C_c^\infty(\R^d)$.
Then
\[
\begin{aligned}
\int_{\R^d} \tilde a \nabla_{x'} G_0 \cdot\nabla_{x'} v \dx'
&= \int_{\omega^c} a(x_0) \nabla_{x'} G_0 \cdot\nabla_{x'} v \dx' + \int_\omega b  \nabla_{x'} G_0\cdot \nabla_{x'} v \dx'\\
&= \int_{\partial \omega^c}a(x_0)\nabla_{x'} G_0\cdot (-x')\, v \dx' + \int_{\omega} b  \nabla_{x'} G_0 \cdot\nabla_{x'} v \dx'\\
&= \int_{\partial \omega^c}a(x_0)(g-d(g\cdot x')x')\cdot (-x')\, v \dx' + \int_{\partial\omega} b  g\cdot x' v \dx'\\
&= \int_{\partial \omega}(b+a(x_0)(d-1)) g\cdot x' v \dx'
\end{aligned}
\]
and
\[
 g \cdot \int_\omega \nabla_{x'} v \dx' = \int_\omega \nabla_{x'}(g\cdot x')\cdot\nabla_{x'}v \dx' = 0 + \int_{\partial\omega} g\cdot x' v \dx.
\]
And $G$ satisfies the integral equation when tested with test functions from $C_c^\infty(\R^d)$.
Let us argue that $G\in \dot H^1(\R^d)$. For $d>1$, we have $|\nabla G| \sim |x|^{-d}$ for $|x|>1$.
For $d=1$, it holds $\nabla G(x)=0$ for $|x|>1$.
It follows $\nabla G \in L^2(\R^d)$.
By density, $G$ satisfies the equation for all test functions.
\end{proof}

\begin{theorem} \label{thm_exp_balls}
Let $a(x_0),b\in \mathcal M$ be multiples of the identity matrix.
Then for almost all $x_0\in \Omega$,
we have
\[
 \delta J(a;b,x_0,B(0,1) ) = -\nabla y(x_0) \cdot \nabla p(x_0) \frac{a(x_0)d}{b+a(x_0)(d-1)} ( b-a(x_0) ).
\]
\end{theorem}
\begin{proof}
Due to \cref{lem_explicit_KQ} and \eqref{eq_Q}, we have for $x'\in \omega$
\[
 Q(x') = (b-a(x_0)) \nabla p(x_0) \cdot x'   \cdot \frac{-1}{b+a(x_0)(d-1)},
\]
which implies
\[
\int_\omega \nabla_{x'} Q_r \dx = (b-a(x_0)) \nabla p(x_0) |\omega| \cdot \frac{-1}{b+a(x_0)(d-1)}.
\]
Using this identity in the result of \cref{thm_exp_general}, we find
\[\begin{split}
  \lim_{r\to0} \frac1{|B(r)|} (J(y_r) - J(y) ) &=-(b-a(x_0))\nabla y(x_0)\cdot \nabla p(x_0) \cdot \left( 1 - \frac{b-a(x_0)}{b+a(x_0)(d-1)}\right) \\
  &=-\nabla y(x_0) \cdot\nabla p(x_0) \frac{a(x_0)d}{b+a(x_0)(d-1)} ( b-a(x_0) ).
\end{split}\]
\end{proof}

This result coincides with those of \cite[Theorem 6.1]{Amstutz2006}, \cite[Theorem 3.1]{CarpioRapun2008},
which were derived under much stronger assumptions.

\subsection{Ellipses}
\label{sec_ellipses}

Now let $H \in \R^{2,2}$ be symmetric, positive definite with $\det H=1$.
Then
\begin{equation}\label{eq_def_ellips}
\omega = \{ x\in \R^2: \ x^THx \le1 \}
\end{equation}
is an ellipse with $|\omega|=|B(0,1)|$.
We will now study perturbations on elliptic sets (instead of on balls). We restrict to the 2d case, where explicit formulas for the polarization matrix are available
from \cite{BruhlHankeVogelius2003, KangKim2007}.
Such explicit formulas are available for $d=3$ as well \cite{KangKim2007}, but are much more technical to analyze.
At the end of the section, we will argue that it is enough to consider elliptic perturbations.

Here, we have the following result  for axis-aligned ellipses.
It was derived in \cite{BruhlHankeVogelius2003} using an explicitly constructed solution to the equation \eqref{eq_Q} in terms of elliptic coordinates.

\begin{lemma}\label{lem_exp_ell_diag}
Let $H =\operatorname{diag}(\lambda, \lambda^{-1}) \in \R^{2,2}$ with $\lambda>0$ be given, and define $\omega$ as in \eqref{eq_def_ellips}.
Let $a(x_0),b\in \mathcal M$ be multiples of the identity matrix.
Then for almost all $x_0\in \Omega$,
we have
\[
 \delta J(a;b,x_0, \omega ) = -(b-a(x_0))  a(x_0) \nabla p(x_0) \cdot
   \begin{pmatrix}
	\frac{ \lambda+1}{a(x_0) \lambda + b }& 0 \\
      0& \frac{ \lambda+1}{a(x_0)  + b \lambda}\\
     \end{pmatrix} \nabla y(x_0).
\]
\end{lemma}
\begin{proof}
From \cref{thm_exp_general}, we get
\[
  \delta J(a;b,x_0, \omega ) = - (b-a(x_0)) \nabla y(x_0) \cdot \left( \nabla p(x_0)  - \int_\omega \nabla_{x'} Q \dx' \right),
\]
where $Q$ solves \eqref{eq_Q}. Using the polarization matrix $M$, cf., \eqref{eq_expansion_polariz}, we can write
\[
    \delta J(a;b,x_0, \omega ) = - \frac1{|\omega|}  (b-a(x_0)) \frac{a(x_0)}b  \nabla y(x_0) \cdot M \nabla p(x_0).
\]
The matrix $M$ was computed in \cite{BruhlHankeVogelius2003}.
Using \cite[eq. (A.4)]{BruhlHankeVogelius2003} with $\kappa:=b$, $\gamma:=a(x_0)$, $a:=\lambda^{1/2}$, $b:=\lambda^{-1/2}$, $|\omega|=1$,
we find
\[
 M
 = |\omega| \begin{pmatrix}
	\frac{\kappa (a+b)}{\gamma a + \kappa b}& 0 \\
      0& \frac{\kappa (a+b)}{\gamma b + \kappa a}\\
     \end{pmatrix}
 = \begin{pmatrix}
	\frac{b (\lambda+1)}{a(x_0) \lambda + b }& 0 \\
      0& \frac{b (\lambda+1)}{a(x_0)  + b \lambda}\\
     \end{pmatrix},
\]
and the claim follows.
\end{proof}

Note that in the case of $H=I_2$ and $\lambda=\lambda^{-1}=1$ this reduces to the result of \cref{thm_exp_balls}.

\begin{corollary}\label{cor_exp_ell_general}
 Let $H =R^T\operatorname{diag}(\lambda, \lambda^{-1})R \in \R^{2,2}$ with $\lambda>0$, and $R^TR=I_2$, and define $\omega$ as in \eqref{eq_def_ellips}.
Let $a(x_0),b\in \mathcal M$ be multiples of the identity matrix.
Then for almost all $x_0\in \Omega$,
we have
\begin{multline*}
 \delta J(a;b,x_0, \omega ) = \\
 -(b-a(x_0))  a(x_0) \nabla p(x_0) \cdot R^T
    \begin{pmatrix}
	\frac{ \lambda+1}{a(x_0) \lambda + b }& 0 \\
      0& \frac{ \lambda+1}{a(x_0)  + b \lambda}\\
     \end{pmatrix} R \nabla y(x_0).
 \end{multline*}
 for almost all $x_0$.
\end{corollary}
\begin{proof}
Define $\omega' := \{x: \ x^T\operatorname{diag}(\lambda, \lambda^{-1})x\le 1\}$,
which implies $\omega = R^T \omega'$.
 If $M_\omega$ and $M_{\omega'}$ are the polarization matrices associated
 to $\omega$ and $\omega'$, then it holds $M_\omega = R^TM_{\omega'} R$ by \cite[Lemma 4.5]{AmmariKang2007}.
 Using the matrix $M$ from the proof of \cref{lem_exp_ell_diag}, the claim follows.
\end{proof}

We will now compute the range of
\[
 \omega \mapsto  \delta J(a;b,x_0, \omega ),
 \]
where $\omega$ varies over elliptic shapes as considered in the above results.
Similar considerations are done in \cite{Amstutz2011b} with a different approach and different notation.

\begin{lemma} \label{lem_bnd_rot}
 Let $y,p\in \R^2$ be given. Let $D=\diag (\lambda_1,\lambda_2)$ be a diagonal matrix.
 For $R\in \mathbb R^{2,2}$ define
 \[
  F(R):= p^T R^T D R y.
 \]
Then the range of $F$ is given by
\[
 F( \Otwo) = F(\SOtwo)= \frac{\lambda_1+\lambda_2}2y^Tp + [-1,+1] \cdot \frac{|\lambda_1-\lambda_2|}2  \|y\|_2\|p\|_2,
\]
where $\Otwo=\{ R\in \R^{2,2}: \ R^TR=I_2\}$, $\SOtwo = \{R\in \Otwo: \det(R)=1\}$.
\end{lemma}
\begin{proof}
The claim is trivially true if $y=0$ or $p=0$.
Hence, it suffices to consider the case $y\ne0$ and $p\ne0$.
We first prove the claim under the assumption $\|y\|_2=\|p\|_2=1$.
The general case follows from a simple scaling argument.
 Let $R\in \Otwo$ with $\det(R)=-1$ be given, and set $S:=\diag(1,-1)$. Then $SR\in \SOtwo$ and $F(R) = F(SR)$ since $SDS=D$.
 Now, define $R_0$ to be the rotation matrix that rotates $p$ onto the first unit vector $e_1$, i.e.,
 \[
  R_0=\begin{pmatrix}
		p_1 & p_2 \\
		-p_2 & p_1
      \end{pmatrix}.
 \]
 Set $\tilde y:=R_0y$. Then it is enough to compute the range of $\tilde F(R):= F(RR_0)$.
 Given $v\in \R^2$ with $\|v\|_2=1$, we parametrize $R$ as follows
 \[
  R(v):=\begin{pmatrix} v_1 & v_2 \\ -v_2 & v_1\end{pmatrix} = v_1 I_2 + v_2 J, \quad
 J := \begin{pmatrix} 0 & 1 \\-1 & 0 \end{pmatrix}.
\]
Then we get
\[\begin{split}
   \tilde F(R(v))& =e_1^T (v_1 I_2 + v_2 J)^T D (v_1 I_2 + v_2 J) \tilde  y \\
   &=v^T \begin{pmatrix}
     \lambda_1\tilde y_1 & \frac{\lambda_1-\lambda_2}2 \tilde y_2\\
     \frac{\lambda_1-\lambda_2}2 \tilde y_2  & \lambda_2 \tilde y_1
    \end{pmatrix} v.
  \end{split} =: v^TMv
\]
Hence, the range of $v \mapsto F(R(v))$ is equal to the interval determined by the eigenvalues of the matrix $M$.
    A short calulation shows that the eigenvalues of $M$ are given by
    \[
       t_{1,2}:=  \frac{\lambda_1+\lambda_2}2 \tilde y_1 \pm \frac{|\lambda_1-\lambda_2|}2.
    \]
Since $\tilde y_1=y^Tp$, the claim follows.
\end{proof}

\begin{lemma}\label{lem_range_of_Gfunction}
 Let $y,p\in \R^2$ and $a,b>0$ be given.
 For $R\in \mathbb R^{2,2}$ and $\lambda>0$ define
 \[
  G(R,\lambda):=
  p^T R^T
   \begin{pmatrix}
	\frac{ \lambda+1}{a \lambda + b}& 0 \\
      0& \frac{ \lambda+1}{a  + b \lambda}\\
     \end{pmatrix} R y.
     \]
  Then
  \[
  \operatorname{cl}
   G( \Otwo,\ \R^+) = \frac12\left(\frac1a + \frac1b\right) y^Tp + [-1,+1] \cdot  \frac12 \left(\frac1a-\frac1b\right) \cdot \|y\|_2\|p\|_2.
  \]
\end{lemma}
\begin{proof}
As in the proof of the previous result, \cref{lem_bnd_rot}, it is enough to consider the case $\|y\|_2=\|p\|_2=1$.
In addition, we can assume $b\ge a$. In case $b<a$, we can consider the function $G(R,\lambda^{-1})$, which is equal to $G(R,\lambda)$ but with the roles of $a,b$ exchanged.
 Define
 \[
  \lambda_1(\lambda) := \frac{ \lambda+1}{a \lambda + b },
  \ \lambda_2(\lambda) := \frac{ \lambda+1}{a  + b \lambda}.
 \]
Note that $\lambda_2(\lambda) = \lambda_1(\lambda^{-1})$.
Since $b>a$ it follows that $\lambda_1$ is monotonically increasing, and $\lambda_2$ is monotonically decreasing.
%
This implies that $|\lambda_1(\lambda)-\lambda_2(\lambda)|$ is maximal at $\lambda=0$ and $\lambda\to \infty$,
with maximum $\frac1a-\frac1b$.

In addition, we find
\[
\lambda_1(\lambda) + \lambda_2(\lambda) = (a+b) \frac{ (\lambda+1)^2 }{ ab\lambda^2 +(a^2+b^2)\lambda + ab}
 \]
with derivative
\[
 \frac\d{\d\lambda}( \lambda_1(\lambda) + \lambda_2(\lambda) ) = (a+b) (a - b)^2 \frac{\lambda^2-1 }{(a \lambda + b)^2 (a + b \lambda)^2}.
\]
Hence, the minimum and maximum of $\lambda_1(\lambda) + \lambda_2(\lambda)$ are attained at $\lambda=1$ and $\lambda=0$, respectively,
which means
\begin{equation}\label{eq_l1pl2}
\lambda_1(1) + \lambda_2(1) =\frac4{a+b}\le \lambda_1(\lambda) + \lambda_2(\lambda) \le \frac1a+\frac1b = \lambda_1(0) + \lambda_2(0) .
\end{equation}
By \cref{lem_bnd_rot}, we have
\begin{equation}\label{eq_GO2l_bnd}
G(\Otwo,\lambda) = \frac{\lambda_1(\lambda)+\lambda_2(\lambda)}2y^Tp + \left[ -\frac{|\lambda_1(\lambda)-\lambda_2(\lambda)|}2,\, +\frac{|\lambda_1(\lambda)-\lambda_2(\lambda)|}2\right].
\end{equation}
Let us only consider the case $y^Tp\ge0$, the case $y^Tp<0$ can be proven by a simple change of sign.
Then using that the supremum of both addends is attained at $\lambda=0$, see \eqref{eq_l1pl2}, we get
\[\begin{split}
\sup_{\lambda>0} \sup (G(\Otwo,\lambda)) & = \sup_{\lambda>0} \frac{\lambda_1(\lambda)+\lambda_2(\lambda)}2y^Tp +\frac{|\lambda_1(\lambda)-\lambda_2(\lambda)|}2 \\
&=  \frac{(\lambda_1(0)+\lambda_2(0))}2y^Tp +\frac{|\lambda_1(0)-\lambda_2(0)|}2 \\
&=  \frac12\left(\frac1a + \frac1b\right) y^Tp + \frac12\left(\frac1a - \frac1b\right).
\end{split}
\]
To compute the infimum, we observe that the lower bound in \eqref{eq_GO2l_bnd} is invariant under the transform $\lambda\mapsto \lambda^{-1}$.
Hence, it is sufficient to consider $\lambda\ge1$.
Here, we find
\[\begin{split}
 \inf_{\lambda>0} \inf (G(\Otwo,\lambda)) &= \inf_{\lambda\ge1} \inf (G(\Otwo,\lambda)) \\
 &= \inf_{\lambda\ge1} \frac{\lambda_1(\lambda)+\lambda_2(\lambda)}2y^Tp -\frac{\lambda_1(\lambda)-\lambda_2(\lambda)}2\\
 &= \frac12\inf_{\lambda\ge1} \Big( \lambda_1(\lambda)  \underbrace{ (y^Tp -1) }_{\le 0} + \lambda_2(\lambda) \underbrace{ (y^Tp +1) }_{\ge0}  \Big) \\
 & = \frac12\lim_{\lambda\to\infty} \big( \lambda_1(\lambda)   (y^Tp -1) + \lambda_2(\lambda) (y^Tp +1) \big) \\
 &= \frac12\left(\frac1a + \frac1b\right) y^Tp - \frac12\left(\frac1a - \frac1b\right)
\end{split}\]
due to the monotonicity properties of $\lambda_1$ and $\lambda_2$.
And the claim is proven.
\end{proof}

Using these results, we can
compute the range of possible values of the topological
derivative when
the shape of the perturbation $\omega$ varies over all possible ellipses.
The infimum in the next result will be useful for necessary optimality conditions.

\begin{theorem}\label{thm_exp_2d}
Let $a(x_0),b\in \mathcal M$ be multiples of the identity matrix.
Let us define
\[\begin{split}
 \mathcal H &:= \{ H \in \R^{2,2} \text{ positive definite  with } \det H=1\}\\
 \omega(H)&:= \{ x: \ x^THx \le 1\}.
\end{split}\]
Then for almost all $x_0\in \Omega$,
the closure of the range of $\delta J$ with respect to variations of ellipses $\omega$ is given by
\begin{multline*}
 \operatorname{cl}
\{  \delta J(a;b,x_0, \omega(H) ) : H\in \mathcal H\} \\
\begin{aligned}
 =&-(b-a(x_0)) \nabla y(x_0) \cdot \nabla p(x_0) \\
 &+ \frac12 \frac{ (b-a(x_0))^2}b \left([-1,+1]\cdot \|\nabla y(x_0)\|_2\|\nabla p(x_0)\|_2 + \nabla y(x_0) \cdot \nabla p(x_0)\right)
.
 \end{aligned}
 \end{multline*}
In particular,
\begin{multline*}
 \inf
\{  \delta J(a;b,x_0, \omega(H) ) : H\in \mathcal H\} \\
\begin{aligned}
 =&-(b-a(x_0)) \nabla y(x_0) \cdot \nabla p(x_0) \\
 &+ \frac12 \frac{ (b-a(x_0))^2}b \left(   \nabla y(x_0) \cdot \nabla p(x_0)- \|\nabla y(x_0)\|_2\|\nabla p(x_0)\|_2\right).
\end{aligned}
\end{multline*}
\end{theorem}
\begin{proof}
First, we apply \cref{cor_exp_ell_general} to matrices $H\in \mathcal H$ with rational eigenvalues and eigenvectors.
Such matrices are dense in $\mathcal H$. Then for almost all $x_0\in \Omega$
it holds
\[
 \operatorname{cl}
\{  \delta J(a;b,x_0, \omega(H) ) : H\in \mathcal H\}
 = -(b-a(x_0))  a(x_0)
  \operatorname{cl}
   G( \Otwo,\ \R^+, x_0),
\]
where $G$ is defined by
\[
  G(R,\lambda,x_0):=
  \nabla p(x_0)^T R^T
   \begin{pmatrix}
	\frac{ \lambda+1}{a(x_0) \lambda + b}& 0 \\
      0& \frac{ \lambda+1}{a(x_0)  + b \lambda}\\
     \end{pmatrix} R \nabla y(x_0).
\]
Due to \cref{lem_range_of_Gfunction}, we get
\begin{multline*}
\operatorname{cl}
   G( \Otwo,\ \R^+, x_0) =\frac12\left(\frac1{a(x_0)} + \frac1b\right) \nabla y(x_0) \cdot \nabla p(x_0)
   \\
   + [-1,+1] \cdot  \frac12 \left(\frac1{a(x_0)}-\frac1b\right) \cdot \|\nabla y(x_0)\|_2\|\nabla p(x_0)\|_2.
\end{multline*}
The claim follows now from elementary computations.
\end{proof}

Let us briefly show that it suffices to study elliptic inclusions.
Here, we will use the results of \cite{CapdeboscqVogelius2004}.
The result of \cref{lem_exp_ell_diag} can be written as
\begin{equation} \label{eq_sens_with_bounds}
 \delta J(a;b,x_0, \omega ) = -(b-a(x_0))  \nabla p(x_0) \cdot M \nabla y(x_0),
\end{equation}
where $M$ is a matrix that depends on $\omega$, $a(x_0)$, $b$.
Note that \eqref{eq_sens_with_bounds} does not contain the factor $a(x_0)$, which is present in the estimate of \cref{cor_exp_ell_general}.
As shown in \cite{CapdeboscqVogelius2004},
for every symmetric matrix $M\in \R^{2,2}$ with eigenvalues in the set
\[
\left\{ ( \lambda_1,\lambda_2): \ \lambda_1 + \lambda_2 \le 1+ \frac{ a(x_0) }b, \quad  \frac1{\lambda_1} + \frac1{\lambda_2} \le 1+ \frac b{ a(x_0) }\right\}
\]
there is an inclusion $\omega$ such that \eqref{eq_sens_with_bounds} is valid.
The bound on the inverses on the eigenvalues is attained for ellipses.
The bound on $\lambda_1 + \lambda_2$ is attained as the limit for differences of ellipses.
As shown in \cite{CapdeboscqVogelius2004}, this set of matrices can be parametrized as
\[
   \begin{pmatrix}
	\frac{ b + a \lambda}{b(1+\lambda)}& 0 \\
      0& \frac{a + b \lambda}{b(1+\lambda)}\\
     \end{pmatrix}, \quad \lambda>0.
\]
When optimizing over the latter set of matrices, we get the same bounds on the range.
Note that we need to take account for the missing factor $a(x_0)$ in \eqref{eq_sens_with_bounds} when compared to \cref{cor_exp_ell_general}.

\begin{lemma}\label{lem_range_of_Gfunction_washers}
 Let $y,p\in \R^2$ and $a,b>0$ be given.
 For $R\in \mathbb R^{2,2}$ and $\lambda>0$ define
 \[
  G(R,\lambda):=
  p^T R^T
   \begin{pmatrix}
	\frac{ b + a \lambda}{ab(1+\lambda)}& 0 \\
      0& \frac{a + b \lambda}{ab(1+\lambda)}\\
     \end{pmatrix} R y.
     \]
  Then
  \[
  \operatorname{cl}
   G( \Otwo,\ \R^+) = \frac12\left(\frac1a + \frac1b\right) y^Tp + [-1,+1] \cdot  \frac12 \left(\frac1a-\frac1b\right) \cdot \|y\|_2\|p\|_2.
  \]
\end{lemma}
\begin{proof}
 Let $\lambda_1(\lambda) :=\frac{ b + a \lambda}{ab(1+\lambda)}$ and $\lambda_2(\lambda):= \frac{a + b \lambda}{ab(1+\lambda)}$.
 Using \cref{lem_bnd_rot}, we get
\[
  \operatorname{cl}
   G( \Otwo,\ \lambda)= \frac{\lambda_1(\lambda)+\lambda_2(\lambda)}2y^Tp + [-1,+1] \cdot \frac{|\lambda_1(\lambda)-\lambda_2(\lambda)|}2  \|y\|_2\|p\|_2.
\]
Here, we have $\lambda_1(\lambda)+\lambda_2(\lambda) = \frac1a + \frac1b$.
In addition, we have $|\lambda_1(\lambda)-\lambda_2(\lambda)| \le \frac{ |a-b|}{ab} $  with $\lim_{\lambda\searrow0} |\lambda_1(\lambda)-\lambda_2(\lambda)| = \frac{ |a-b|}{ab} $,
which proves the claim.
\end{proof}

\section{Control in the coefficients}
\label{sec_pmp}

In this section, we will apply the results on the $\delta J$ to derive the Pontryagin maximum principle for the following
optimization problem: Minimize
\begin{equation}\label{eq_def_ctrl}
\tilde  J(y,a) := \frac12 \|y-y_d\|_{L^2(\Omega)}^2 + \int_\Omega g( a(x)) \dx,
\end{equation}
subject to $a\in \mathcal A$, and $y\in H^1_0(\Omega)$ solves
\[
 \int_\Omega a \nabla y \cdot \nabla v \dx = \int_\Omega f v\dx \quad \forall v\in H^1_0(\Omega).
\]
This is a classical problem, and lead to the study of H-convergence, \cite{MuratTartar1997}.
Solutions do not exist in general.
The data of the problem is assumed to satisfy the basic \cref{ass_basic}.
In addition, $g:\R^{d,d} \to \R \cup \{+\infty\}$ is lower semi-continuous, which implies that $g(a)$ is measurable for all $a\in \mathcal A$.

We will formulate necessary optimality conditions in terms of the Pontryagin maximum principle.
Using the results from the previous sections, we find that the maximum principle holds in the following form.

\begin{theorem}\label{thm_pmp}
 Let $a$ be a local solution of \eqref{eq_def_ctrl} with respect to the $L^1(\Omega)$-norm.
 Let $y$ and $p$ be the corresponding solution of the state equation \eqref{eq_def_y} and \eqref{eq_def_p}.
 Let $\omega \subset \R^d$ be open and bounded with $0\in \omega$.
 Then for almost all $x_0\in \Omega$ and all $b\in \mathcal M$ (see \eqref{eq_def_M}) it holds
 \[
 -(b-a(x_0)) \nabla y(x_0) \cdot \left(\nabla p(x_0) + \frac1{|\omega|} \int_\omega \nabla_{x'} Q \dx' \right)
 +
  g(b) - g(a(x_0)) \ge 0,
 \]
where $Q\in \dot H^1(\R^d)$ is the solution of \eqref{eq_Q}.
\end{theorem}
\begin{proof}
Fix $b\in \mathcal M$ and $x_0 \in \Omega$. Define $a_r$ as in \eqref{eq_def_ar} and $y_r$ as in \eqref{eq_def_yr}. Then $a_r \to a $ in $L^1(\Omega)$ for $r\searrow0$.
Hence, $\tilde J(y_r,a_r) - \tilde J(y,a) \ge0$ for all $r>0$ small enough.
Using \cref{thm_exp_general} and \cref{thm_lebesgue_lp}, we get
\[\begin{split}
 0 &\le \lim_{r\searrow0} \frac1{r^d|\omega|}(\tilde J(y_r,a_r) - \tilde J(y,a) ) \\
 &=
 -(b-a(x_0)) \nabla y(x_0) \cdot \left(\nabla p(x_0) + \frac1{|\omega|} \int_\omega \nabla_{x'} Q(x_0,b) \dx' \right)
 + g(b) - g(a(x_0)),
\end{split}\]
where the limit exists for almost all $x_0\in \Omega$.
Here, $Q(x_0,b) \in \dot H^1(\R^d)$ denotes the solution of \eqref{eq_Q_b} for given $b\in \mathcal M$ and $x_0 \in \Omega$.

This argument shows that for each $b\in \mathcal M$ there exists a set of zero measure $N_b$ such that
for all $x \in \Omega \setminus N_b$
\[
  0 \le
 -(b-a(x_0)) \nabla y(x_0) \cdot \left(\nabla p(x_0) + \frac1{|\omega|} \int_\omega \nabla_{x'} Q(x_0,b) \dx' \right)
 + g(b) - g(a(x_0)).
\]
Let $G$ be a  countable and dense subset of $\gph g = \{ (b,g(b)): b \in \mathcal M\}$.
We write $G = \{ (b_k, g(b_k)): \ k \in \N\}$.
Then  for all $k$ there exists a set of zero measure $N_k$
such that for all $x \in \Omega \setminus N_k$
\[
  0 \le
 -(b_k-a(x_0)) \nabla y(x_0) \cdot \left(\nabla p(x_0) + \frac1{|\omega|} \int_\omega \nabla_{x'} Q(x_0,b_k) \dx' \right)
 + g(b_k) - g(a(x_0)).
\]
Let $N:= \bigcup_{k=1}^\infty N_k$, which has zero measure.
Then for all $x\in \Omega \setminus N$ and all $k$ the inequality
\[
  0 \le
 -(b_k-a(x_0)) \nabla y(x_0) \cdot \left(\nabla p(x_0) + \frac1{|\omega|} \int_\omega \nabla_{x'} Q(x_0,b_k) \dx' \right)
 + g(b_k) - g(a(x_0))
\]
is satisfied.
For fixed $x_0 \in \Omega$
the mapping $ b \mapsto Q(x_0,b)$ is continuous from $\mathcal M$ to $\dot H^1(\R^d)$ by \cref{lem_cont_b_to_Q}.
By density, the desired inequality is satisfied for all $x\in \Omega\setminus N$ and all $b\in \mathcal M$.
\end{proof}

\subsection{The scalar case}

Now let us investigate the case when $a$ is a multiple of the identity.
That is, we will prove necessary optimality condition for the problem
\begin{equation}\label{eq_def_ctrl_scalar}
 \min \tilde J(y,a) \text{ subject to } a(x) \in \mathcal M \cap \operatorname{span}(I_d) \text{ f.a.a. } x.
\end{equation}
With little abuse of notation, we will consider now $a:\Omega \to \R$ and $b\in \R$.
In this section, we will use ball-shaped perturbations.
Using the explicit expression from \cref{thm_exp_balls}, we get the following form of the maximum principle:

\begin{theorem}\label{thm_pmp_scalar}
 Let $a$ be a local solution of \eqref{eq_def_ctrl_scalar} with respect to the $L^1(\Omega)$-norm.
 Let $y$ and $p$ be the corresponding solutions of the state equation \eqref{eq_def_y} and adjoint equation \eqref{eq_def_p}.
 Then for almost all $x_0\in \Omega$ and all $b\ge \alpha$ (see \eqref{eq_def_M}) it holds
 \begin{equation}\label{eq_pmp_scalar}
 -(b-a(x_0))\nabla y(x_0) \cdot \nabla p(x_0) \frac{a(x_0)d}{b+a(x_0)(d-1)}
+  g(b) - g(a(x_0)) \ge 0.
\end{equation}
\end{theorem}
\begin{proof}
 The proof is similar to that of \cref{thm_pmp} and uses \cref{thm_exp_balls} for the expression of $\delta J$.
\end{proof}

\subsection{The scalar case, $d=2$}

In the two-dimensional case, we have the following expression for the maximum principle associated with \eqref{eq_def_ctrl_scalar},
which is based on the formulas using ellipses.

\begin{theorem}\label{thm_pmp_scalar2d}
 Let $a$ be a local solution of \eqref{eq_def_ctrl_scalar} with respect to the $L^1(\Omega)$-norm.
 Let $y$ and $p$ be the corresponding solutions of the state equation \eqref{eq_def_y} and adjoint equation \eqref{eq_def_p}.
 Then for almost all $x_0\in \Omega$ and all $b\ge \alpha$ it holds
 \begin{multline} \label{eq_pmp_scalar2d}
-(b-a(x_0)) \nabla y(x_0) \cdot \nabla p(x_0)+  g(b) - g(a(x_0)) \\
+ \frac12 \frac{ (b-a(x_0))^2}b \left(   \nabla y(x_0) \cdot \nabla p(x_0)- \|\nabla y(x_0)\|_2\|\nabla p(x_0)\|_2\right)
 \ge 0.
 \end{multline}
\end{theorem}
\begin{proof}
Follows from \cref{thm_exp_2d}.
\end{proof}

An immediate consequence is the following condition, which implies a subdifferential inclusion.
This is unexpected, as we do not put any convexity conditions on $g$.

\begin{corollary}
 Let $a$ be feasible for \eqref{eq_def_ctrl_scalar}.
 Let $y$ and $p$ be the corresponding solutions of the state equation \eqref{eq_def_y} and adjoint equation \eqref{eq_def_p}.
 Suppose that for almost all $x_0\in \Omega$ and all $b\ge \alpha$ the inequality \eqref{eq_pmp_scalar2d}
 is satisfied.
 Then for almost all $x_0\in \Omega$ it holds
\[
 \nabla y(x_0) \cdot \nabla p(x_0) \in \partial (g + I_{[\alpha,\beta]})(a(x_0)),
\]
where $\partial (g + I_{[\alpha,\beta]})$ denotes the convex subdifferential of the possibly non-convex function $g + I_{[\alpha,\beta]}$,
and $I_{[\alpha,\beta]}$ is the indicator function of $[\alpha,\beta]$.
\end{corollary}

In addition, the conclusion of \cref{thm_pmp_scalar2d} is stronger than that of \cref{thm_pmp_scalar}.
This follows from the optimization considerations in \cref{sec_ellipses}.
Here, we will provide a simple, elementary proof.

\begin{lemma}
\label{lem_sanity_check}
 Let $a,b \ge \alpha$, $y,p\in \R^2$.
Then
\[
  -(b-a) \frac{2a}{b+a} y \cdot p
- \left(
-(b-a)  y \cdot  p
+ \frac12 \frac{ (b-a)^2}b \left(    y \cdot  p- \| y\|_2\| p\|_2\right) \right)
\ge0
\]
\end{lemma}
\begin{proof}
Collecting the factors of $y \cdot p$ in the left-hand side of the claimed inequality gives
\[
  -(b-a) \frac{2a}{b+a}
 - \left(
-(b-a)
+ \frac12  \frac{ (b-a)^2}b
\right)
=\frac{(b-a)^3 }{2b(a+b)}.
\]
Using $|b-a| \le a+b$, we can estimate
\[
 \left| \frac{(b-a)^3 }{2b(a+b)} y\cdot p\right|
 \le \frac{(b-a)^2}{2b} \| y\|_2\| p\|_2,
\]
which proves the claimed inequality.
\end{proof}

\begin{corollary}
 Let $a$ be feasible for \eqref{eq_def_ctrl_scalar}.
 Let $y$ and $p$ be the corresponding solutions of the state equation \eqref{eq_def_y} and adjoint equation \eqref{eq_def_p}.
 Suppose that for almost all $x_0\in \Omega$ and all $b\ge \alpha$ the inequality \eqref{eq_pmp_scalar2d}
 is satisfied.
 Then \eqref{eq_pmp_scalar} is true with $d=2$
 for almost all $x_0\in \Omega$ and all $b\ge \alpha$.
 \end{corollary}
\begin{proof}
Combining  \cref{lem_sanity_check} with \cref{eq_pmp_scalar2d} proves
\begin{multline*}
 -(b-a(x_0))\nabla y(x_0) \cdot \nabla p(x_0) \frac{2a(x_0)}{b+a(x_0)} \\
 \begin{aligned}
 &\ge
-(b-a(x_0)) \nabla y(x_0) \cdot \nabla p(x_0) \\
&\qquad
 + \frac12 \frac{ (b-a(x_0))^2}b \left(   \nabla y(x_0) \cdot \nabla p(x_0)- \|\nabla y(x_0)\|_2\|\nabla p(x_0)\|_2\right)
\end{aligned}
\\
 \ge -( g(b) - g(a(x_0)) ),
 \end{multline*}
which proves \eqref{eq_pmp_scalar} in the case $d=2$.
\end{proof}

\subsection{Relation to Fr\'echet derivative}

The mapping $a\mapsto y$ is Fr\'echet differentiable from $L^\infty(\Omega)$ to $H^1_0(\Omega)$.
If $a$ is a local minimum of \eqref{eq_def_ctrl_scalar} and $g$ is continuously differentiable,
then $a$ satisfies the necessary optimality condition
\begin{equation}\label{eq_noc_frechet}
 (-\nabla y(x_0) \cdot \nabla p(x_0)+g'(a))(b-a) \ge0 \quad \forall b\ge \alpha.
\end{equation}
Naturally, the results of \cref{thm_pmp_scalar} and \cref{thm_pmp_scalar2d} are stronger:
replacing $b$ by $a(x_0) + t(b-a(x_0))$ in those inequalities, dividing by $t$, and passing to the limit $t\searrow 0$ yields \eqref{eq_noc_frechet}.

\subsection{Example with linear $g$}

Let us consider the following example, which is motivated by material optimization problems, \cite{Amstutz2011b}.
We consider the feasible set given by
\[
\mathcal M := \{ a I_d : \ a\in [\alpha,\beta]\},
\]
where $0<\alpha<\beta$ are real numbers.
In addition, we chose
\[
 g(a) := \ell \cdot a,
\]
where $\ell\in \R$, which could be used to model the cost of materials.
In this simplified situation, we can analyze the optimality conditions \eqref{eq_pmp_scalar2d} and \eqref{eq_noc_frechet}.
Solutions of \eqref{eq_noc_frechet} can be characterized by the implications
\begin{align*}
 \ell > \nabla y(x_0) \cdot \nabla p(x_0) \ \Rightarrow\ a(x_0) = \alpha \ \Rightarrow\ \ell - \nabla y(x_0) \cdot \nabla p(x_0) \ge0, \\
 \ell < \nabla y(x_0) \cdot \nabla p(x_0) \ \Rightarrow\ a(x_0) = \beta  \ \Rightarrow\ \ell - \nabla y(x_0) \cdot \nabla p(x_0) \le0.
\end{align*}
This  should be compared to the characterization of solutions of the condition \eqref{eq_pmp_scalar2d},
which is stated next.

\begin{corollary}
 Let $a$ be feasible for \eqref{eq_def_ctrl_scalar}.
 Let $y$ and $p$ be the corresponding solutions of the state equation \eqref{eq_def_y} and adjoint equation \eqref{eq_def_p}.

 Suppose that for almost all $x_0\in \Omega$ and all
 $b\in[\alpha,\beta]$
 the inequality \eqref{eq_pmp_scalar2d}
 is satisfied.
 Then for almost all $x_0$ the following implications hold:
 \begin{align*}
  \ell = \nabla y(x_0) \cdot \nabla p(x_0) &\ \Rightarrow\ \|\nabla y(x_0)\|_2\|\nabla p(x_0)\|_2 = \nabla y(x_0) \cdot \nabla p(x_0)\\
  \ell > \nabla y(x_0) \cdot \nabla p(x_0) &\ \Rightarrow\ a(x_0) = \alpha \\
  \ell < \nabla y(x_0) \cdot \nabla p(x_0) &\ \Rightarrow\ a(x_0) = \beta.
  \end{align*}
  In addition, if $  a(x_0) = \alpha $ then
  \begin{multline*}
  \ell - \nabla y(x_0) \cdot \nabla p(x_0) \\
  \ge \frac12 \frac{ \beta-\alpha}\beta  (\|\nabla y(x_0)\|_2\|\nabla p(x_0)\|_2-\nabla y(x_0) \cdot \nabla p(x_0)) \ge0.
  \end{multline*}
  If $a(x_0) = \beta$ then
  \begin{multline*}
  \ell - \nabla y(x_0) \cdot \nabla p(x_0) \\
  \le \frac12 \frac{\alpha- \beta}\alpha (\|\nabla y(x_0)\|_2\|\nabla p(x_0)\|_2-\nabla y(x_0) \cdot \nabla p(x_0)) \le0.
  \end{multline*}

\end{corollary}
\begin{proof}
Let us define for abbreviation
\[
 s(x_0):=  \nabla y(x_0) \cdot \nabla p(x_0), \quad n(x_0) :=  \|\nabla y(x_0)\|_2\|\nabla p(x_0)\|_2,
\]
which implies $|s| \le n$. Then \eqref{eq_pmp_scalar2d}
is equivalent to
\[
-(b-a(x_0)) s(x_0)+   \ell(b - a(x_0)) \\
+ \frac12 \frac{ (b-a(x_0))^2}b (s(x_0)-n(x_0))
 \ge 0
\]
and
\begin{equation}\label{eq_var_ineq_ell}
(b - a(x_0))(\ell - s(x_0)) \\
\ge \frac12 \frac{ (b-a(x_0))^2}b (n(x_0)-s(x_0)) \quad \forall b\in[\alpha,\beta].
\end{equation}
Now let us assume that \eqref{eq_pmp_scalar2d} is true, and hence the inequality is satisfied for almost all $x_0$.
In case, $\ell - s(x_0) = 0$ it follows $n(x_0)=s(x_0)$.

If $\ell - s(x_0) \ne 0$ it follows $a(x_0) \in \{\alpha,\beta\}$ as the left-hand side of \eqref{eq_var_ineq_ell}
changes sign at $b=a(x_0)$, while the right-hand side is non-negative.
Suppose $a(x_0) = \alpha$. Then \eqref{eq_var_ineq_ell} implies
\[
 \ell - s(x_0) \ge \frac12 \frac{ \beta-\alpha}\beta  (n(x_0)-s(x_0)).
\]
If $a(x_0) = \beta$ we get the reverse inequality
\[
  \ell - s(x_0) \le \frac12 \frac{ \alpha-\beta}\alpha (n(x_0)-s(x_0)).
\]
\end{proof}

\section{Conclusion and outlook}

We developed the Pontryagin maximum principle for control in the coefficients using quite elementary methods.
It would be interesting to consider more complicated settings using general cost functionals  semilinar or quasilinear equations.
Also the case $b=0$ could be considered following \cite{Amstutz2021}.
Another interesting question is, whether the Ekeland variational principle could be used to prove existence of $\epsilon$-solutions of the Pontryagin
maximum principle.
The maximum principle in the $2d$ case was written in terms of a variational inequality \eqref{eq_pmp_scalar2d} of a new type,
whose solution theory is completely open.

\printbibliography

\end{document}